\documentclass{article}
\usepackage{amssymb}
\usepackage{graphicx}
\usepackage{amsmath}
\usepackage{epsfig}
\usepackage{color}

\newtheorem{theorem}{Theorem}

\newtheorem{lemma}[theorem]{Lemma}

\newtheorem{remark}[theorem]{Remark}

\newenvironment{proof}[1][Proof]{\textbf{#1.} }{\ \rule{0.5em}{0.5em}}

\begin{document}
			
			\title{ Helicity  in dispersive fluid mechanics}
			\author{S. L. Gavrilyuk$^{\ast }$  and H. Gouin\thanks{Aix Marseille University, CNRS, IUSTI, UMR 7343, Marseille, France.\newline  E-mails: sergey.gavrilyuk@univ-amu.fr; henri.gouin@ens-lyon.org;henri.gouin@univ-amu.fr}}
			
			\maketitle

			\begin{abstract}
				\noindent  By dispersive models of fluid mechanics  we are referring to the Euler-Lagrange equations for the  constrained Hamilton action functional  where  the internal  energy depends on high order derivatives of unknowns. The mass conservation law is considered as a constraint.  The corresponding Euler-Lagrange equations  include, in particular,  the van der Waals--Korteweg model  of capillary fluids, the model of fluids containing small gas bubbles and  the model describing long free-surface gravity waves. We obtain new conservation laws  generalizing  the helicity conservation  for classical barotropic fluids.  
			\end{abstract}
			

			\section{Introduction}

			The helicity integral was first discovered for classical barotropic fluids in seminal papers by  J. J. Moreau \cite{Moreau} and   H. K.  Moffatt   \cite{Moffatt_1}. As the Kelvin circulation theorem,  conservation of helicity results from the particle re-labeling symmetry   \cite{Gouin_1976,Salmon1988,Salmon, Araki_2015}.  Further important studies of such  invariants as well as   Ertel's type  invariants for the incompressible and compressible Euler equations, magnetohydrodynamics etc.  can be found in the literature \cite{Tur_1993,Arnold,Cotter_2013,Webb_2014_II,Cheviakov_2014,Moffatt_2,Irvine,D.Serre}.
			
			  In this paper, we derive helicity   integrals   for dispersive media such as   fluids  endowed with capillarity (called also  Euler--van der Waals--Korteweg's fluids,  and a class of fluids with internal inertia (such as  bubbly fluids and  long  free-surface gravity waves). 
			
 From a mathematical point of view,  	we   concentrate here on two classes of  dispersive equations which are Euler–Lagrange equations for Hamilton’s
			action with a Lagrangian depending not only on the thermodynamic variables but also on their first spatial and time derivatives.
			
			 The first class of the models (the Lagrangian depends on the density
			gradient) is often called {\it the second gradient models}, and the corresponding dispersive terms are called {\it capillary type dispersion terms}.  Indeed, the density is linked directly to the determinant of the deformation gradient, so the density gradient thus linked to second derivatives of the deformation. Historically, such models appeared in the description of  regions with non-uniform   density  and described the surface tension effects    due to the van der Waals forces between  fluid molecules. The dependence on the  density gradient describes  in some sense  the micro-structure of such a non-homogeneous continuum. Such class of models thus describes a continuum with  a characteristic micro-scale length \cite{Casal1,Eglit, Truskinovsky,Casal2,Gavrilyuk_Shugrin_1996,Benzoni-Gavage,Garajeu_Gouin_Saccomandi,Gouin_Saccomandi,Haspot_2018,Bresch,Gavrilyuk3}. 
			
			For the second class of models, the Lagrangian depends on the material
			derivative  of the density (the dependence on the material derivative, and not just  on the partial time derivative,  is important to guarantee the Galilean invariance of the governing equations). This type of dispersion is often called {\it  inertia type dispersion}. Such class of models  contains thus  a characteristic micro-scale time  \cite{Salmon,Iordansky,Kogarko,Wijngaarden,Gavrilyuk_Teshukov,Gavrilyuk_2011,Serre_1953, Green_Naghdi_2,Lannes_2013}.
		
	These two classes of  models can obviously be unified in a general setting \cite{Gavrilyuk2}. However, for  specific applications (Euler--van der Waals--Korteweg fluids, fluids containing small gas bubbles, etc. ) the governing equations describe  dispersive effects separately: they take into account either the micro-scale structure or a micro-inertia structure. We will therefore study these two classes of models separately. \\

			We introduce the following notations. 	For any vectors $\boldsymbol{a}$ and $\boldsymbol{b}$, the scalar product  $\boldsymbol{a}\cdot \boldsymbol{b}$ is written   $\boldsymbol{a}^{T}\boldsymbol{b}$  
			(line vector or covector $\boldsymbol{a}^{T}$
			is multiplied by column vector $\boldsymbol{b}$), where superscript $^T$ denotes the transposition. The tensor product $\boldsymbol{a}\otimes \boldsymbol{b}$  is also  written $\boldsymbol{a} { \ }\boldsymbol{b}^{T}$.

			The product of   matrix  $\boldsymbol A$   by  vector   ${\boldsymbol a}$
			is denoted by   $\boldsymbol A\,  {\boldsymbol a} $. Notation   ${\boldsymbol
				b}^T \, \boldsymbol A\, $ means the covector   ${\boldsymbol c}^T $ defined
			as ${\boldsymbol c}^T = ( \boldsymbol A^T\,  {\boldsymbol b})^T$. The mixed  product of the vectors ${\boldsymbol a}$, ${\boldsymbol b}$, ${\boldsymbol c}$   is denoted $ \,\ \boldsymbol a^T(\boldsymbol b \times\boldsymbol c)={\boldsymbol a}\cdot(\boldsymbol b \times            \boldsymbol c)={\rm det}({\boldsymbol a},{\boldsymbol b},{\boldsymbol c})$\,\ where the sign `$\times$'  means the vector product. Tensor $\boldsymbol{I}$ denotes the identity transformation.
			
			The gradient of  scalar function $f(\boldsymbol x)$ is defined  by  $\nabla f=\displaystyle\left(\frac{\partial f}{\partial
				{\boldsymbol x}}\right)^T$ associated with  linear form $df=\displaystyle  \displaystyle\left(\frac{\partial f}{\partial
				{\boldsymbol x}}\right)\,d{\boldsymbol x}$.
			
			Let $\boldsymbol v=\left(v^1,...,v^n\right)^T$ be a function of  $\boldsymbol x=\left(x^1,...,x^n\right)^T$. Then\  $\displaystyle\frac{\partial  \boldsymbol v} {\partial  \boldsymbol x}$\ denotes the linear application defined by the relation $\displaystyle d\boldsymbol v= \frac{\partial \boldsymbol v} {\partial \boldsymbol x}\, d\boldsymbol x$ and represented   by the  matrix
			\begin{equation*}
				\displaystyle
				\left(
				\begin{array}{ccc}
					\displaystyle\frac{\partial v^{1}(x^{1},...,x^{n})}{\partial x^{1}}  & ... &
					\displaystyle\frac{\partial v^{1}(x^{1},...,x^{n})}{\displaystyle\partial
						x^{n}} \\
					\vdots &  & \vdots \\
					\displaystyle\frac{\partial v^{n}(x^{1},...,x^{n})}{\partial x^{1}} & ... & %
					\displaystyle\frac{\partial v^{n}(x^{1},...,x^{n})}{\partial x^{n}}%
				\end{array}%
				\right) .
			\end{equation*}\\
			
			The
			divergence of second order  tensor   $\boldsymbol A$   is the
			covector  $ {\rm div}\, \boldsymbol A  $ such that, for any constant vector
			${\boldsymbol a}$, $ ({\rm div}\, \boldsymbol  A)\, {\boldsymbol a}    =   {\rm
				div }\, (\boldsymbol A\ {\boldsymbol a})$. It implies that for any vector
			field $\boldsymbol{v}$  
			\begin{equation*}
				\text{div}( {\boldsymbol A}\,\boldsymbol{v})=
				(\text{div }{\boldsymbol A})\,\boldsymbol{v}+ \text{Tr}\left({\boldsymbol A}\,\dfrac{\partial
					\boldsymbol{v}}{\partial \boldsymbol{x}}\right),
			\end{equation*}
			where $\text{Tr}$ is the trace operator. \\
			
			In three-dimensional case  we introduce vorticity  $\boldsymbol \omega= {\rm curl}\,  \boldsymbol u$. In the  Cartesian basis  $\boldsymbol \omega=(\omega^1,\omega^2, \omega^3)^T$ is defined by
			\begin{equation*}
				\frac{\partial \boldsymbol  u}{\partial \boldsymbol {x}}-	\left(\frac{\partial \boldsymbol  u}{\partial \boldsymbol x}\right)^T=\displaystyle
				\left(
				\begin{array}{ccc}
					\displaystyle 0  & -\omega^3 &\omega^2\\
					\omega^3 & 0 & -\omega^1 \\
					-\omega^2 &\omega^1 & 0
				\end{array}%
				\right).
			\end{equation*}
			
%

			\section{Fluid motion}

			A fluid motion is 
			represented by  a diffeomorphism  $\boldsymbol{\varphi}$ of a
			three-dimensional reference configuration $\mathcal D_0$ into the physical space $\mathcal D_t$. Let 
			$\boldsymbol{X}=(X^{1},X^{2},X^{3})^T$  be  the Lagrangian coordinates,  and $\boldsymbol{x}=(x^{1},x^{2},x^{3})^T$ be  the Eulerian coordinates. Fluid motion is   
			\begin{equation*}
			\boldsymbol{x}= {\boldsymbol\varphi}\left(t, \boldsymbol{X}\right),
				\label{motion}
			\end{equation*}
			where $t$ denotes the time.  At a given time $t$, the transformation ${\boldsymbol\varphi}$
			possesses an inverse and has continuous derivatives up to the second
			order. The deformation gradient is  
			\begin{equation*}
				{\boldsymbol F}=\frac{\partial  {\boldsymbol\varphi(t,\boldsymbol X)}}{\partial \boldsymbol X} \equiv\frac{\partial  {\boldsymbol x}}{\partial \boldsymbol X}.
				\label{deformation_gradient}
			\end{equation*}
			Considering $ {\boldsymbol F}$ as a function of the Eulerian coordinates, we obtain  the evolution equation 
			\begin{equation}
				\frac{D{\boldsymbol F}}{Dt}= \frac{\partial \boldsymbol u}{\partial \boldsymbol x}{\boldsymbol F} \quad {\rm with}\quad\frac{D}{D t}=\frac{\partial}{\partial t}+{\boldsymbol u}^T \nabla,\label{derF}
			\end{equation}
			where $\displaystyle\boldsymbol u (t, \boldsymbol x)=\left.\frac{\partial  {\boldsymbol\varphi(t,\boldsymbol X)}}{\partial t}\right\vert_{\boldsymbol X=\boldsymbol\varphi^{-1}(t, \boldsymbol x)}$ is  the particle velocity.
			
			\section{Basic lemmas}
			\begin{lemma}
				\label{lemma_1}
				The following identities are satisfied:
				\begin{equation}
					{\rm div}\left(\frac{\boldsymbol F}{{\rm det}\, \boldsymbol{F}}\right)=0,
					\label{first}
				\end{equation}
				\begin{equation}
					\frac{D}{Dt}\left(\frac{\boldsymbol F}{\rm det\, \boldsymbol F}\right)= \left(\frac{\partial \boldsymbol{u}}{\partial \boldsymbol{x}}-({\rm div}\, \boldsymbol u)\,\boldsymbol{I}\right)\, \frac{\boldsymbol F}{\rm det\, \boldsymbol F}.
					\label{second}
				\end{equation}
			\end{lemma}   
			
			\begin{proof}
	We consider a vector field $\boldsymbol v_0$ on $\mathcal D_0$ of boundary $  S_0$ and their images  $\mathcal D_t$ of boundary  $S_t$. We obtain
				\begin{equation*}
					\iiint_{\mathcal D_t}{\rm div}\left(\frac{\boldsymbol F}{{\rm det}\, \boldsymbol{F}}\, \boldsymbol v_0\right) dD=\iiint_{\mathcal D_0}({\rm det}\,\boldsymbol F)\,{\rm div}\left(\frac{\boldsymbol F}{{\rm det}\, \boldsymbol{F}}\, \boldsymbol v_0\right) dD_0.
				\end{equation*}
		Let $\boldsymbol X\left(u_{10},u_{20}\right)$ be the parametrization of the surface $S_0$ in the reference space, and $\boldsymbol x\left(u_{1},u_{2}\right)$ the parametrization of its image, $S_t$, in the actual space.  \\
			We denote by $\displaystyle d_{10}\boldsymbol X=\frac{\partial\boldsymbol X}{\partial u_{10}}\; du_{10}$, $\displaystyle d_{20}\boldsymbol X=\frac{\partial\boldsymbol X}{\partial u_{20}}\; du_{20}$  the  infinitesimal tangent  vectors  along the  coordinate lines  on $S_0$, and $\displaystyle d_{1}\boldsymbol x=\frac{\partial\boldsymbol x}{\partial u_{1}}\; du_{1}$, $\displaystyle d_{2}\boldsymbol x=\frac{\partial\boldsymbol x}{\partial u_{2}}\; du_{2}$ the  infinitesimal tangent vectors  along the corresponding material lines on $ S_t$.\\
		For any constant vector $\boldsymbol a$,
		\begin{equation*}
			\boldsymbol a^T(d_1\boldsymbol x\times d_2\boldsymbol x)={\rm det}\left(\boldsymbol a, d_1\boldsymbol x, d_2\boldsymbol x
			\right)={\rm det}\left(\boldsymbol a, \boldsymbol F d_{10}\boldsymbol X, \boldsymbol F d_{20}\boldsymbol X
			\right)
			\end{equation*}
			\begin{equation*}
			= ({\rm det}\boldsymbol F)\,\,{\rm det}\left(\boldsymbol F^{-1}\boldsymbol a,d_{10}\boldsymbol X, d_{20}\boldsymbol X\right)= ({\rm det}\boldsymbol F)\,\,\left(\boldsymbol F^{-1}\boldsymbol a\right)^T\left(d_{10}\boldsymbol X\times d_{20}\boldsymbol X\right).
		\end{equation*}
		This gives us the relationship obtained in \cite{Gurtin_Fried_Anand_2007}, p. 77 :
		\begin{equation}
			d_{1}\boldsymbol x\times d_{2}\boldsymbol x = ({\rm det}\boldsymbol F)\, \boldsymbol F^{-T}	\left(d_{10}\boldsymbol X\times d_{20}\boldsymbol X\right).\label{key1}
		\end{equation}
		From Eq. \eqref{key1} and Gauss--Ostrogradsky's formula, we obtain 
		\begin{equation*}
			\iiint_{\mathcal D_t}{\rm div}\left(\frac{\boldsymbol F}{{\rm det}\, \boldsymbol{F}}\, \boldsymbol v_0\right) dD= \iint_{ S_t} {\rm det}\left(\frac{\boldsymbol F}{{\rm det}\, \boldsymbol{F}}\, \boldsymbol v_0, d_1\boldsymbol x, d_2\boldsymbol x\right)
		\end{equation*}
		\begin{equation*}
		 =\iint_{ S_0} {\rm det}\left(\frac{\boldsymbol F}{{\rm det}\, \boldsymbol{F}}\, \boldsymbol v_0,{\boldsymbol F}\, d_{10}\boldsymbol X, {\boldsymbol F}\,d_{20}\boldsymbol X\right) 
		\end{equation*}
		\begin{equation*}
			= \iint_{ S_0}{\rm det}\left(\boldsymbol v_0,d_{10}\boldsymbol X,d_{20}\boldsymbol X\right)  =\iiint_{\mathcal D_0}({\rm div_0}\,{\boldsymbol v_0})\ dD_0. 
		\end{equation*}\\
	
\noindent Consequently,

	\begin{equation*}
		({{\rm det}\, \boldsymbol{F}}) \,{\rm div}\left(\frac{\boldsymbol F}{{\rm det}\, \boldsymbol{F}}\, \boldsymbol v_0\right)={\rm div}_0\, \boldsymbol v_0.
	\end{equation*}	
For all constant vector $\boldsymbol v_0$ on  ${\mathcal D_0}$, we get  \begin{equation*}{\rm div}\left(\frac{\boldsymbol F}{{\rm det}\, \boldsymbol{F}}\, \boldsymbol v_0\right) = {\rm div}\left(\frac{\boldsymbol F}{{\rm det}\, \boldsymbol{F}}\right)\,\, \boldsymbol v_0 = 0.
\end{equation*}
This implies \eqref{first}.\\

\noindent Finally, 	from \eqref{derF}   we obtain 
				\begin{equation*}
					\frac{D}{Dt}\left(\frac{\boldsymbol F}{{\rm det }\,\boldsymbol{F}}\right) = \frac{1}{{\rm det }\,\boldsymbol{F}}\frac{D{\boldsymbol F}}{Dt}-\frac{1}{({\rm det }\,\boldsymbol{F})^2}\,\frac{D({\rm det }\,\boldsymbol{F})}{Dt}\,\boldsymbol{F}.
				\end{equation*}
			The Euler-Jacobi   formula 
				\begin{equation*}
					\frac{D({\rm det}\,{\boldsymbol F})}{Dt}=({\rm det}\,{\boldsymbol F})\,{\rm div}{\boldsymbol{u}}
				\end{equation*}
			gives \eqref{second}.
			\end{proof}\\
		
			A simple  geometrical  interpretation of the identity \eqref{first} is given below.  
\begin{lemma}
	\label{lemma_2}
	Let $\boldsymbol{e}_i$ be a local curvilinear  basis (with lower indexes $i=1,2,3$ ),  $\boldsymbol{e}_i=\displaystyle\frac{\partial \boldsymbol{x}}{\partial X^i}=\frac{\partial \boldsymbol{\varphi}(t,\boldsymbol{X})}{\partial X^i}$,  expressed in the Eulerian coordinates, and $\boldsymbol{e}^i=\nabla X^i(t, \boldsymbol{x})$ (with upper indexes  $i=1,2,3$) be  the corresponding dual basis  (cobasis). Then 
	\begin{equation*}  
		\frac{\boldsymbol{e}_k}{{\rm det}\,{\boldsymbol F}}={\boldsymbol{e}^i}\times{\boldsymbol{e}^j}, 
	\end{equation*}
	where $\{i,j,k\}$ is  an even permutation of  $\{1,2,3\}$. Moreover, 
	\begin{equation*}  
		{\rm div }\left(\frac{\boldsymbol{e}_k}{{\rm det}\,{\boldsymbol F}}\right)=0. 
	\end{equation*}
\end{lemma}
\begin{proof}
	The proof of the first formula comes from the identity  
	\begin{equation*}
\delta^i_j=	\frac{\partial X^i}{\partial X^j}=\frac{\partial X^i}{\partial \boldsymbol x}\frac{\partial \boldsymbol x}{\partial X^j}=	{\boldsymbol{e}^i}^T{\boldsymbol{e}_j},
	\end{equation*}
where $\delta_j^i$ is the  Kronecker symbol. The second formula comes from the identity 
	\begin{equation*}  
		{\rm div }\left(\frac{\boldsymbol{e}_k}{{\rm det}\,{\boldsymbol F}}\right)={\rm div }\left({\boldsymbol{e}^i}\times{\boldsymbol{e}^j}\right)
	\end{equation*}
		\begin{equation*}  
	={\boldsymbol{e}^j}^T{{\rm curl}\;  {\boldsymbol e}^i}-{\boldsymbol{e}^i}^T{{\rm curl}\;  {\boldsymbol e}^j}=0.
\end{equation*}	
\end{proof}
%
%
 \begin{remark}
\noindent 	
 Let $\boldsymbol E_k =\displaystyle \frac{\boldsymbol{e}_k}{{\rm det}\,{\boldsymbol F}}$  be   the columns  of the matrix 
	$\displaystyle \frac{\boldsymbol F}{{\rm det }\,\boldsymbol{F}}$ :
	\begin{equation*}
		\frac{\boldsymbol F}{{\rm det }\,\boldsymbol{F}}=\left[\boldsymbol E_1,\boldsymbol E_2,\boldsymbol E_3\right]. 
	\end{equation*}
The equations \eqref{first}, \eqref{second} yield :
	\begin{equation}
		\frac{\partial {\boldsymbol{E}}_i}{\partial t}+ {\rm curl}\left({\boldsymbol{E}}_i\times {\boldsymbol{u}}\right)=0, \quad {\rm div}\,{\boldsymbol{E}}_i=0,\quad i=1,2,3.
		\label{E_equation}
	\end{equation}
This is   Helmholtz's equation which governs   the  vorticity in the   Euler equations.  
	Equations  \eqref{E_equation} express that the  Lie derivative of a two--form along the velocity  field $\boldsymbol u$ is zero (see Appendix \ref{Appendix A} for details) : 
		\begin{equation*}
d_{L_2}	{\boldsymbol E}_i\equiv\frac{D{\boldsymbol E}_i}{Dt}-\frac{\partial \boldsymbol{u}}{\partial \boldsymbol{x}}{\boldsymbol E}_i+({\rm div}\, \boldsymbol u)\;{\boldsymbol E}_i=0.
	\end{equation*}
 \end{remark}
\begin{lemma}
	\label{lemma_identity}
For any scalar field $G$ and   vector  fields $\boldsymbol K$,    ${\boldsymbol L}$  on $\mathcal D_t$  
such that ${\boldsymbol L}$ is divergence-free \,  (${\rm div}\,{\boldsymbol L}=0$),
	we have the identity  
	\begin{equation*}
		\begin{array}{ccc}
			\displaystyle	\frac{\partial}{\partial t}\left({\boldsymbol K}^T \,{\boldsymbol L}\right)+{\rm div}\left\{\left({\boldsymbol u}\, {\boldsymbol K}^T  + \left( G-\frac{|\boldsymbol u|^2}{2}\right)\boldsymbol I\right){\boldsymbol L}\right\}\equiv\\ \\
			\displaystyle \boldsymbol L^T  \left(\frac{D\boldsymbol K}{Dt}+ \left(\frac{\partial \boldsymbol u}{\partial\boldsymbol x}\right)^T\boldsymbol K+ \nabla\left(G- \frac{|\boldsymbol u|^2}{2}\right)\right)+\\ \\ \displaystyle \boldsymbol K^T\left(\frac{\partial \boldsymbol L}{\partial t}+\frac{\partial \boldsymbol L}{\partial \boldsymbol x}\,\boldsymbol u+ \boldsymbol L\,{\rm div}\,\boldsymbol u-\frac{\partial \boldsymbol u}{\partial \boldsymbol x}\,\boldsymbol L \right).
			\label{helicity4}
		\end{array}
	\end{equation*}
\end{lemma}
\begin{proof}
	Let us denote 
	\begin{equation*}
		A= \displaystyle	\frac{\partial}{\partial t}\left({\boldsymbol K}^T \,{\boldsymbol L}\right)+{\rm div}\left\{\left({\boldsymbol u}\, {\boldsymbol K}^T  + \left( G-\frac{1}{2} \, {\boldsymbol u}^T {\boldsymbol u}\right)\boldsymbol I\right){\boldsymbol L}\right\}.
	\end{equation*}
	From 
	\begin{equation*}
		{\rm div}\left(\boldsymbol u\, \boldsymbol K^T\boldsymbol L \right)= \boldsymbol K^T\boldsymbol L \left({\rm div}\,\boldsymbol u\right)+\boldsymbol K^T\frac{\partial \boldsymbol L}{\partial \boldsymbol x}\,\boldsymbol u+\boldsymbol L^T\frac{\partial \boldsymbol K}{\partial \boldsymbol x}\,\boldsymbol u 
	\end{equation*}
and 
	\begin{equation*}{\rm div \,{\boldsymbol L}}=0, \quad 	\frac{\partial\boldsymbol K}{\partial t}=\frac{D\boldsymbol K}{Dt}-\frac{\partial\boldsymbol K}{\partial\boldsymbol x}\, \boldsymbol u, 
	\end{equation*}   
we obtain  
	\begin{equation*}
		\begin{array}{ccc}
			\displaystyle A= {\boldsymbol K}^T \frac{\partial\boldsymbol L}{\partial t} + {\boldsymbol L}^T \frac{D\boldsymbol K}{Dt}-{\boldsymbol L}^T\frac{\partial \boldsymbol K}{\partial \boldsymbol x}\,\boldsymbol u+\boldsymbol K^T\boldsymbol L \left({\rm div}\,\boldsymbol u\right)\\ \\ \displaystyle+\boldsymbol K^T\frac{\partial \boldsymbol L}{\partial \boldsymbol x}\,\boldsymbol u+\boldsymbol L^T\frac{\partial \boldsymbol K}{\partial \boldsymbol x}\,\boldsymbol u+\frac{\partial}{\partial\boldsymbol x}\left( G-\frac{|\boldsymbol u|^2}{2}\right)\boldsymbol L.
		\end{array}
	\end{equation*}
Finally, from 
	\begin{equation*}{\boldsymbol L}^T\left(\frac{\partial \boldsymbol u}{\partial \boldsymbol x}\right)^T 	{\boldsymbol K}= {\boldsymbol K}^T\,\left(\frac{\partial \boldsymbol u}{\partial \boldsymbol x}\right)	{\boldsymbol L}
	\end{equation*}
we get  
	\begin{equation*}
		\begin{array}{ccc}
			\displaystyle A= {\boldsymbol L}^T \left(\frac{D\boldsymbol K}{Dt}+\left(\frac{\partial \boldsymbol u}{\partial \boldsymbol x}\right)^T	{\boldsymbol K} +\nabla\left( G-\frac{|\boldsymbol u|^2}{2}\right)\right)\\ \\
			\displaystyle +\,\boldsymbol K^T\left(\frac{\partial\boldsymbol L}{\partial t}+\frac{\partial\boldsymbol L}{\partial\boldsymbol x}\,\boldsymbol u+\boldsymbol L \left({\rm div}\,\boldsymbol u\right)-\frac{\partial\boldsymbol u}{\partial\boldsymbol x}\,\boldsymbol L \right).
		\end{array}
	\end{equation*}
\end{proof}
\begin{remark}
The equation
	\begin{equation}
	\frac{D\boldsymbol K}{Dt}+\left(\frac{\partial \boldsymbol u}{\partial \boldsymbol x}\right)^T	{\boldsymbol K} +\nabla\left( G-\frac{|\boldsymbol u|^2}{2}\right)=0,
	\label{K_equation}
\end{equation}
 it equivalent to 
\begin{equation*}
	d_{L_1}\boldsymbol K^T+\frac{\partial }{\partial \boldsymbol x}\left(G-\frac{|\boldsymbol u|^2}{2}\right)=0
\end{equation*}
where  
$\displaystyle
	d_{L_1}\boldsymbol K^T\equiv\frac{D{\boldsymbol K}^T}{Dt}+{\boldsymbol K}^T\left(\frac{\partial \boldsymbol u}{\partial \boldsymbol x}\right)$	
 is the Lie derivative of   covector ${\boldsymbol K}^T$ along the vector field $\boldsymbol u$ (see Appendix \ref{Appendix A} for details).
\end{remark} 			
\begin{theorem}
	\label{basic_theorem} 
	$1^o)$
 Consider a divergence--free  field  $\boldsymbol{L}$  satisfying the Helmholtz equation 
 \begin{equation*}
d_{L_2}	{\boldsymbol L}\equiv\frac{\partial \boldsymbol L}{\partial t}+\frac{\partial \boldsymbol L}{\partial \boldsymbol x} {\boldsymbol u}+{\boldsymbol L}\,{\rm div} \,{\boldsymbol u}-\frac{\partial \boldsymbol u}{\partial \boldsymbol x} {\boldsymbol L}=0,  
\label{L_equation}
\end{equation*}
 and the field $\boldsymbol K$ satisfying \eqref{K_equation}. Then,
	\begin{equation}
	\frac{\partial}{\partial t}\left({\boldsymbol K}^T \,{\boldsymbol L}\right)+{\rm div}\left\{\left({\boldsymbol u}\, {\boldsymbol K}^T  + \left( G-\frac{|\boldsymbol u|^2}{2}\right)\boldsymbol I\right){\boldsymbol L}\right\}=0.
	\label{K_L}
\end{equation}
$2^o)$  Consider  a material domain $\mathcal D_t$ of boundary $S_t$.  If  at $t=0$ the   divergence-free field  $\boldsymbol{L}$  is tangent to $S_0$, then   at all time  $\boldsymbol{L}$ is tangent to $S_t$.\\ Moreover,  the quantity
	\begin{equation}
	{\mathcal H}=\iiint_{D_t}\boldsymbol{K}^T\boldsymbol{L} \, dD,
	\label{helicity_integral}
	\end{equation}
	   we call {\it generalized helicity}, keeps a constant value along the  motion.
\end{theorem}
The conservation law \eqref{K_L} is a consequence of Lemma \ref{lemma_identity}. The helicity integral \eqref{helicity_integral} immediately comes from \eqref{K_L}.
\\
In particular, the theorem \ref{basic_theorem} implies the following compact form of the conservation laws in the case   $\boldsymbol L=\boldsymbol E_i$,  \,$i=1,2,3$ 
\begin{equation}
		\displaystyle	\frac{\partial}{\partial t}\left(\boldsymbol K^T\frac{\boldsymbol F}{{\rm det\, \boldsymbol F }}\right)+{\rm div}\left\{\left(\boldsymbol u\,\boldsymbol K^T+\left(G-\frac{|\boldsymbol u|^2}{2}\right)\boldsymbol I\right)\frac{\boldsymbol F}{{\rm det\, \boldsymbol F }}\right\}=\boldsymbol 0^T.
	\label{graph1}
\end{equation}
 The analog of conservation law \eqref{graph1}  has already been found   in the modeling of multiphase flows  \cite{Gavrilyuk_Gouin_Perepechko_1998}. 
\section{Euler--van der Waals--Korteweg's  fluids}
We consider compressible fluids endowed with internal capillarity when  the energy per  unit  volume  is a function  of density  $\rho$ and $\nabla\rho$ \cite{Casal1,Eglit,Truskinovsky,Casal2,Gavrilyuk_Shugrin_1996,Gavrilyuk2,Benzoni-Gavage,Haspot_2018,Bresch}. These fluids are  also called Euler--van der Waals--Korteweg's  fluids. For barotropic fluids, the  volume energy  is in the form
\begin{equation*}
W= W(\rho,\nabla\rho).
\end{equation*}  
The governing equations are the Euler-Lagrange equations  associated with the Hamilton action 
\begin{equation*}
a=\int_{t_1}^{t_2}\int_{D_t} \mathcal L\; dt\,dD,  
\end{equation*}
where $t_1, \; t_2$ are given time instants.
 The Lagrangian is 
\begin{equation}
\mathcal{L}=\frac{\rho \vert\boldsymbol u\vert^2}{2}-W(\rho,\nabla\rho)-\rho\,V(t, \boldsymbol x).
\label{EvWK}
\end{equation}
Here $V(t,\boldsymbol x)$  is the potential of external forces.   The mass conservation law
 \begin{equation}
 \frac{\partial \rho}{\partial t}+{\rm div}(\rho\, \boldsymbol{u})=0
 \label{mass}
 \end{equation} 
 is  a constraint. The momentum equation is  the  Euler-Lagrange equation for \eqref{EvWK} 
\begin{equation*}
\frac{\partial	\rho\, \boldsymbol{u}^T}{\partial t}+\text{div}  \left(\rho\,\boldsymbol{u}\otimes \boldsymbol{u}+\boldsymbol{\Pi}\right)+ \rho\,\frac{\partial V}{\partial \boldsymbol x}= \boldsymbol 0^T, \quad {\boldsymbol \Pi}=P\,\boldsymbol I
+\,\dfrac{\partial W}{\partial\nabla\rho}\otimes \nabla\rho \label{motion capillary}
\end{equation*}
 where    pressure  term $P$  is defined as 
\begin{equation*}
P=\rho\frac{\delta W}{\delta\rho}-W \quad {\rm with}\quad \frac{\delta W}{\delta\rho}=\frac{\partial W}{\partial \rho}-{\rm div}\left(\frac{\partial  W}{\partial \nabla \rho}\right) . 
\end{equation*}
The term $\delta W/\delta \rho$\,\ is the variational derivative of $W$ with respect to $\rho$.
The energy conservation law is 
\begin{equation*}
\frac{\partial	e}{\partial t}+ {\rm div}\left(e\boldsymbol  u+{\boldsymbol\Pi} \boldsymbol  u-\dfrac{\partial\rho}{\partial t}\dfrac{\partial W}{\partial\nabla\rho}\right)-\rho\,\frac{\partial V}{\partial t}=0, \label{energy capillary}
\end{equation*}
with the definition of the total volume energy $e$ 
\begin{equation*}
e=\frac{\rho\vert\boldsymbol  u\vert^2}{2}+W+\rho\,V.
\end{equation*}
Due to the mass conservation law \eqref{mass}, the momentum equation can  be written  
\begin{equation*}
\frac{D \boldsymbol u}{D t}+\nabla\left(\frac{\delta W}{\delta \rho}+V\right) =\boldsymbol 0. \label{momentum1}
\end{equation*}
We  use further the notation 
\begin{equation*}
H=\frac{\delta W}{\delta \rho}+V 
\end{equation*}
for the total generalized specific enthalpy. Obviously, the momentum equation can be written in the form  of Eq.  \eqref{K_equation} with $\boldsymbol K=\boldsymbol u$  
\begin{equation}
	\frac{D \boldsymbol u}{D t}+ \left(\frac{\partial \boldsymbol u}{\partial \boldsymbol x}\right)^T \boldsymbol u+\nabla\left(H-\frac{\vert\boldsymbol u\vert^2}{2}\right) =\boldsymbol 0.
	\label{capillary_noncons}
\end{equation}
We denote the vorticity of the capillary fluid by 
\begin{equation*}
\boldsymbol\omega = {\rm curl}\,\boldsymbol u.
\end{equation*}
The vorticity satisfies the Helmholtz equation 
\begin{equation*}
\frac{\partial \boldsymbol \omega}{\partial t}+\frac{\partial \boldsymbol \omega}{\partial \boldsymbol x} {\boldsymbol u}+{\boldsymbol \omega}\,{\rm div} \,{\boldsymbol u}-\frac{\partial \boldsymbol u}{\partial \boldsymbol x} {\boldsymbol \omega}=0.  \label{vorticity}
\end{equation*} 
Applying the theorem  \ref{basic_theorem},  we obtain 
\begin{theorem}
	$1^o)$
Equations of capillary fluids \eqref{capillary_noncons} admit the following conservation law  
	\begin{equation*}
	\frac{\partial}{\partial t}\left({\boldsymbol u}^T \,{\boldsymbol \omega}\right)+{\rm div}\left\{\left({\boldsymbol u}^T{\boldsymbol \omega}\right){\boldsymbol u}+ \left(H-\frac{\vert\boldsymbol u\vert^2}{2}\right){\boldsymbol \omega}\right\}=0\label{helicity1}
	\end{equation*}
$2^o)$ Consider a material domain $D_t$ of boundary $S_t$. If   at $t=0$ the vector field  $\boldsymbol{\omega}$ is tangent to $S_0$, then  for any time $t$ the vector field $\boldsymbol{\omega}$ is tangent to $S_t$ and  the helicity
\begin{equation*}
	{\mathcal H}=\iiint_{D_t}\boldsymbol{u}^T\boldsymbol{\boldsymbol \omega} \, dD 
\end{equation*}
keeps a constant value along the  motion.\\ The results remain true if we replace $\boldsymbol \omega$ by   vectors $\displaystyle\boldsymbol E_i$, $i=1,2,3$.
\end{theorem} 
\medskip
\noindent The proof of the theorem follows immediately from theorem \ref{basic_theorem} with  $\boldsymbol K = \boldsymbol u$ and $\boldsymbol L = \boldsymbol\omega$ or $\boldsymbol L= \boldsymbol E_i$, $i=1,2,3$.
\section{Fluids with internal inertia}
We consider dispersive models  with \textit{internal inertia}  including a non-linear one-velocity model of    bubbly  fluids   with incompressible liquid phase  at small volume concentration of gas bubbles (Iordansky--Kogarko--van Wijngarden's model) \cite{Iordansky,Kogarko,Wijngaarden}, and Serre--Green--Naghdi's  equations (SGN eqations) describing long free--surface gravity waves \cite{Salmon,Serre_1953,Green_Naghdi_2}.

 As usually, the mass conservation law \eqref{mass} is a constraint. 
 The momentum equation is the  Euler-Lagrange equation  for the constrained action functional  \cite{Gavrilyuk_Shugrin_1996,Gavrilyuk2,Gavrilyuk_2011} 
\begin{equation*}
a=\int_{t_1}^{t_2}\int_{D_t} {\mathcal L}\; dt\,dD,   
\end{equation*}
where the  Lagrangian is 
\begin{equation*}
{\mathcal L}=\frac{\rho \vert\boldsymbol u\vert^2}{2}-W\left(\rho,\frac{D\rho}{Dt}\right)-\rho\,V(t,\boldsymbol x).
\end{equation*}
Here  $\displaystyle W\left(\rho,\frac{D\rho}{Dt}\right)$ is  a  potential depending not only on the density, but also on the  material time derivative  of the density. We call {\it internal inertia effect} such a dependence on the material time derivative.\\ The conservative form of the momentum equation  is the Euler--Lagrange equation 
\begin{equation}
\frac{\partial	\rho\, \boldsymbol{u}^T}{\partial t}+\text{div}  \left(\rho\,\boldsymbol{u}\otimes \boldsymbol{u}+p\,\boldsymbol I
\right)+\rho \frac{\partial V}{\partial \boldsymbol x}= \boldsymbol 0^T,  \label{inertia effect}
\end{equation}
where pressure $p$ is  defined as 
\begin{eqnarray}
p=\rho\,\frac{\delta W}{\delta\rho}-W\qquad\qquad \qquad\qquad\notag\\
{\rm with}\notag\qquad\qquad \qquad\qquad\qquad\qquad \qquad\qquad\qquad\qquad \qquad\qquad\\
 \frac{\delta W}{\delta\rho}=\frac{\partial W}{\partial \rho}-\frac{\partial }{\partial t}\left(\frac{\partial W}{\partial \displaystyle\left(\frac{D\rho}{Dt}\right)}\right)-{\rm div} {\left(\frac{\partial W}{\partial \displaystyle\left(\frac{D\rho}{Dt}\right)}\,\boldsymbol u\right)}.
\label{pressure_inertia}\
\end{eqnarray}
The energy conservation law is 
\begin{equation*}
\frac{\partial	e}{\partial t}+ {\rm div}\left(\left(e+p\right)\boldsymbol  u\right)-\rho\, \frac{\partial V}{\partial t}=0,  \label{energy bubbly}
\end{equation*}
with  the total volume energy $e$ 
\begin{equation*}
e=\frac{\rho\vert\boldsymbol  u\vert^2}{2}+E+\rho V\quad {\rm where}\quad E=W-\frac{D\rho}{Dt} \,\left(\frac{\partial W}{\partial \displaystyle\left(\frac{D\rho}{Dt}\right)}\right). 
\end{equation*}
The momentum and conservation laws are consequences of the invariance of   Lagrangian under space and time shifts. The momentum equation can also be written as    
\begin{equation*}
\frac{D\boldsymbol  u}{Dt}+\frac{\nabla p}{\rho}+\nabla V=\boldsymbol 0.
\label{momentum_inertia}
\end{equation*} 
One introduces  the vector field $\boldsymbol K$ defined as ( see \cite{Gavrilyuk_Teshukov} )   
\begin{equation}
{\boldsymbol K} ={\boldsymbol u}+ \frac{\nabla \sigma}{\rho} \qquad {\rm where}\qquad \sigma =-\rho\,\left(\frac{\partial W}{\partial \displaystyle\left(\frac{D\rho}{Dt}\right)}\right). 
\label{K_vector}
\end{equation}
Considering  the volume internal energy $E(\rho,\tau)$ as a partial Legendre transform of $W(\rho,\dot\rho)$ 
\begin{equation*}
E(\rho,\tau)=W-\frac{D\rho}{Dt}\, \left(\frac{\partial W}{\partial \displaystyle\left(\frac{D\rho}{Dt}\right)}\right)=W+\frac{D\rho}{Dt}\,\tau   \quad {\rm where}\quad \tau =- \left(\frac{\partial W}{\partial \displaystyle\left(\frac{D\rho}{Dt}\right)}\right),
\end{equation*}
and defining  $\tilde E(\rho, \sigma)$ as $\tilde E(\rho,\sigma)=E(\rho,\sigma/\rho)$,  
the momentum equation becomes \cite{Gavrilyuk_Teshukov} 
\begin{equation}
 \frac{D\boldsymbol  K}{D
 	t}+ \left(\frac{\partial \boldsymbol u}{\partial \boldsymbol x}\right)^T	{\boldsymbol K} +\nabla\left(\tilde E_\rho+V-\frac{\vert\boldsymbol u\vert^2}{2} \right)  =0, 
 \label{momentum30}
\end{equation}
with $\displaystyle \tilde E_\rho = \frac{\partial \tilde E}{\partial \rho}(\rho,\sigma)$. Equation \eqref{momentum30} is in the form \eqref{K_equation}.\\

\noindent The notion of generalized vorticity for bubbly fluids was introduced in \cite{Gavrilyuk_Teshukov}. This generalized vorticity is   defined by 
\begin{equation*}
	\boldsymbol\Omega = {\rm curl}\,\boldsymbol K.
\end{equation*}
It satisfies the Helmholtz equation
\begin{equation*}
	\frac{\partial \boldsymbol \Omega}{\partial t} 
	+ {\rm curl}\,(\boldsymbol\Omega\times {\boldsymbol u})=\boldsymbol 0. 
	\label{Helmholtz}
\end{equation*}
Thus, we obtain 
\begin{theorem}
	\label{bubbly_fluids}
$1^o)$	The equations of fluids with internal inertia   \eqref{inertia effect} admit the conservation law 
	\begin{equation*}
	\frac{\partial}{\partial t}\left({\boldsymbol K}^T \,{\boldsymbol \Omega}\right)+{\rm div}\left\{\left({\boldsymbol K}^T{\boldsymbol \Omega}\right){\boldsymbol u}+ \left( \tilde E_\rho+V-\frac{\vert\boldsymbol u\vert^2}{2} \right){\boldsymbol \Omega}\right\}=0\label{helicity2}
	\end{equation*}
$2^o)$ Consider a material domain $D_t$ of boundary $S_t$. If at $t=0$ the vector field  $\boldsymbol{\Omega}$ is tangent to $S_0$, then  for any time $t$ the vector field $\boldsymbol{\Omega}$ is tangent to $S_t$ and the helicity
\begin{equation*}
	{\mathcal H}=\iiint_{D_t}\boldsymbol{K}^T\boldsymbol{\boldsymbol \Omega} \, dD 
\end{equation*}
 keeps a constant value along the  motion.\\ The results remain true if we replace $\boldsymbol \Omega$ by the vectors $\boldsymbol E_i$, $i=1,2,3$.
\end{theorem}
The proof is an immediate application of Theorem \ref{basic_theorem}.
\\

Now, we consider applications to the Serre--Green--Naghdi  equations.
The SGN  equations were derived  in \cite{Salmon,Serre_1953,Green_Naghdi_2}. The equations represent a two--dimensional asymptotic reduction of  full free--surface Euler equations for  long gravity waves. A mathematical justification of this model is presented  in \cite{Lannes_2013, Duchene_Israwi_2018}. Its numerical study can be found  in \cite{GNsarah,Lannes_Marche_2016, Favrie_Gavrilyuk_2017,Duchene_Klein_2022}. The corresponding potential $W$ is  given in 
\cite{Salmon,Gavrilyuk_Shugrin_1996,Gavrilyuk_Teshukov,Gavrilyuk_2011,Camassa_1997}  
\begin{equation*}
W(h,\dot{h})=\frac{g\,h^2}{2}-\frac{h}{6}\left(\frac{Dh}{Dt}\right)^2  \quad{\rm with}\quad \frac{Dh}{Dt} =\frac{\partial h}{\partial t}+\nabla h\cdot{\boldsymbol u},
\end{equation*}
where  $g$ is the  constant gravity acceleration, $h$ is the  water depth and $\boldsymbol u $ is  the depth averaged horizontal velocity. For the SGN  equations, we have only to  replace  $\rho$ by the  water depth $h$  in   potential  $\displaystyle W\left(\rho,\frac{D\rho}{Dt}\right)$.  The water depth $h$ and depth averaged velocity $\boldsymbol u$ are functions of  time $t$ and of  the horizontal  coordinates ${\boldsymbol x}=(x^1,x^2)^T$. The vector  $\boldsymbol K$ defined by  \eqref{K_vector} can be written in the form
\begin{equation*}
{\boldsymbol K}={\boldsymbol u}-\frac{1}{h}\,\nabla \left( h\,\frac{\partial W}{\displaystyle\partial \left(\frac{Dh}{Dt}\right)} \right)={\boldsymbol u}-\frac{1}{3h}\nabla \left(h^2\,\frac{Dh}{Dt} \right)={\boldsymbol u}+\frac{1}{3h}\nabla \left(h^3\, {\rm div}{\boldsymbol u }\right).
\end{equation*}
The physical meaning of  ${\boldsymbol K}$ is quite interesting  : ${\boldsymbol K}$ is  the fluid velocity tangent to  the free surface \cite{Gavrilyuk_2015, Matsuno_2016}.
Since $\displaystyle\sigma= \frac{h^2}{3}\frac{Dh}{Dt}$, the energy expressed in terms of $h$ and $\sigma $ is 
\begin{equation*}
\tilde E(h,\sigma)=W-\frac{Dh}{Dt}\, \left(\frac{\partial W}{\displaystyle\partial \left(\frac{Dh}{Dt}\right)}\right)=
\frac{g\,h^2}{2}+\frac{9\,\sigma^2}{6h^3}
\end{equation*}
Hence
\begin{equation*}
\tilde E_h(h,\sigma)=g\,h-\frac{9\,\sigma^2}{2\,h^4}=g\,h-\frac{1}{2}\,\left(\frac{Dh}{Dt}\right)^2 
\end{equation*}
and the equation for  $\boldsymbol K$ is 
\begin{equation*}
\frac{D\boldsymbol K}{Dt}+ \left(\frac{\partial \boldsymbol u}{\partial \boldsymbol x}\right)^T	{\boldsymbol K} +\nabla\left(g\,h-\frac{1}{2}\left(\frac{Dh}{Dt}\right)^2 -\frac{\vert\boldsymbol u\vert^2}{2} \right) ={\boldsymbol 0}.
\end{equation*}
 \begin{figure}
	\begin{center}
		\includegraphics[scale=1]{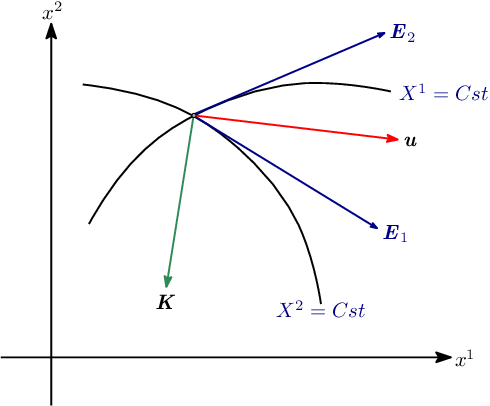}
	\end{center}
	\caption{The  material  curves  $X^1(t, x^1,x^2)=X^{10}$ and $X^2(t, x^1,x^2)=X^{20}$ corresponding to coordinate lines $X^1=X^{10} =Cst$ and $X^2=X^{20}= Cst$ are drawn. At any point, these curves are tangent to the vectors $\boldsymbol E_2$ and $\boldsymbol E_1$.  {\it A priori},  vectors $\boldsymbol u$ and $\boldsymbol K$  are not collinear.}
	\label{Fig.2}
\end{figure}
If theorem \ref{bubbly_fluids} is applied with the vorticity vector $\boldsymbol \Omega$, the result is trivial because  $\boldsymbol \Omega$ and $\boldsymbol K$ are orthogonal.  However, the result is not trivial if we use, instead of vector $\boldsymbol \Omega$, the column vectors ${\boldsymbol E}_i,$ $i=1,2$, of the matrix $\displaystyle {\boldsymbol F}/{{\rm det\, \boldsymbol F }}$   
	\begin{equation*}
	\frac{\partial}{\partial t}\left(\boldsymbol K^T\boldsymbol{E}_i\right)+{\rm div}\left(\boldsymbol u\,\left(\boldsymbol K^T{\boldsymbol E}_i\right)+\left(g\,h-\frac{1}{2}\left(\frac{Dh}{Dt}\right)^2-\frac{|\boldsymbol u|^2}{2}\right){\boldsymbol E}_i\right)=0\label{E1E2}.
\end{equation*}
This equation can be interpreted as  conservation laws for the covariant components of $\boldsymbol K$ in the basis  $\boldsymbol E_i$, $i=1,2$   (see Fig \ref{Fig.2}). 


\section{Conclusion and discussion}
We have found new conservation laws for the equations of dispersive fluids where the energy depends on the density and  gradient of density (capillary fluids), or depends on the density and material time  derivative of the density (fluids with internal inertia). These conservation laws keep  constant  generalized helicity.   

We have also found new  integrals representing the variation of  the components of velocity $\boldsymbol u$ (for capillary fluids) or of  vector $\boldsymbol K$ (for fluids with internal inertia)  in a time--dependent  curvilinear basis ${\boldsymbol E}_i$, $i=1,2,3$. 	These divergence--free basis vectors are tangent at each time  to   material  curves along which coordinate  $X^i$ varies and other coordinates $X^j,\,\ j\neq i$ are fixed.

Other invariants can be found in the case when the potential $W$ depends on additional scalar transported by the fluid such as  the specific entropy $\eta$. Indeed, $\eta$ satisfies the equation 
\begin{equation}
\frac{D\eta}{Dt}=0.
\label{entropy}
\end{equation}  
The analog of Ertel's invariant is  
\begin{equation*}
\frac{D}{Dt}\left(\frac{\partial \eta}{\partial \boldsymbol x}\frac{\boldsymbol L}{\rho}\right)=0
\end{equation*}
where  $\boldsymbol L$ is  either the scaled basis vectors $\boldsymbol E_i$,  $i=1,2,3$\,\ or the {\it generalized} vorticity vector. Indeed, even if the equation for the  vorticity contains an extra term related to the entropy gradient, it vanishes by combining  the Helmholtz equation for the {\it generalized} vorticity with the  Lie derivative of   covector $\displaystyle \frac{\partial \eta}{\partial \boldsymbol x}$  coming from  \eqref{entropy} (see Appendix \ref{Appendix A} for details) 
\begin{equation*}
d_{L_1} \left(\frac{\partial \eta}{\partial \boldsymbol x}\right)\equiv\frac{D}{Dt}\left(\frac{\partial \eta}{\partial \boldsymbol x}\right)+\frac{\partial \eta}{\partial \boldsymbol x}\frac{\partial \boldsymbol u}{\partial \boldsymbol x}=0.
\end{equation*}
In the same way, higher order conservation laws can be found  as it was already done for the classical dispersionless equations (Euler equations of compressible and incompressible fluids, equations of magnetohydrodynamics, two-fluid plasma equations etc.) \cite{Tur_1993,Cotter_2013,Webb_2014_II,Cheviakov_2014}. However, the physical and geometrical meaning of these conservation laws still remains to be done. 
 \medskip
 
 {\bf Acknowledgments} The authors thank Boris Haspot for helpful  comments and discussions. S.G.  would like to thank the Isaac Newton Institute for Mathematical Sciences, Cambridge, for support and hospitality during the programme ``Dispersive hydrodynamics: mathematics, simulation and experiments, with applications in nonlinear waves'' where work on this paper was undertaken. This work was supported by EPSRC grant no EP/R014604/1.

 \appendix
 \section{Lie  derivatives of one-form and two-form fields}
 \label{Appendix A}
 
In the literature, one generally uses the notation $\mathcal L_{\boldsymbol v}$ where ${\boldsymbol v}$ is a vector field. However, the analytical expression of the Lie derivative is different for scalars, one--forms, two--forms and so on. This is why we denote $d_{L_1}$ for the Lie derivative of a one--form and $d_{L_2}$ for the Lie derivative of a two--form \cite{Gouin 2023}.\\

We denote by  $T_{\boldsymbol x} (\mathcal D_t)$ and $ T_{\boldsymbol x}^\star (\mathcal D_t)$ the tangent and cotangent vector spaces at any point $\boldsymbol x$ of a material domain $\mathcal D_t$. A vector field $\boldsymbol J \in T_{\boldsymbol x} (\mathcal D_t)$ has the inverse image $\boldsymbol J_0 \in T_{\boldsymbol X} (\mathcal D_0)$ such that 
\begin{equation}
	\boldsymbol J(t,\boldsymbol x)= \boldsymbol F\, \boldsymbol J_0(t,\boldsymbol X)\label{vectorfield}.
\end{equation}
 Consequently, a one-form field $\boldsymbol C \in T_{\boldsymbol x}^\star (\mathcal D_t)$  has the inverse image $\boldsymbol C_0 \in T_{\boldsymbol X}^\star (\mathcal D_0)$  such that 
\begin{equation*}
\boldsymbol C(t,\boldsymbol x)=  \boldsymbol C_0(t,\boldsymbol X)\, \boldsymbol F^{-1} \label{formfield}.
\end{equation*}
The Lie derivative of $\boldsymbol C$ is  the image of\,\ $\displaystyle {\partial\boldsymbol C_0(t,\boldsymbol X)}/{\partial t}$\,\ in $T_{\boldsymbol x}^\star (\mathcal D_t)$ and is denoted by $d_{L_1}$. From the evolution equation \eqref{derF} for $\boldsymbol F$,  the
Lie derivative $d_{L_1}$ of  form field $\boldsymbol C$ is deduced from the
diagram 
\begin{equation*}
	\begin{array}{ccc}
		\boldsymbol C\in T^{\ast }_{\boldsymbol x}      (\mathcal{D}_{t}) &
		\begin{array}{c}
		{\boldsymbol F}
		  \\  
			\longrightarrow \\
			\\
		\end{array}
		& \boldsymbol C_0=\boldsymbol C\,\boldsymbol F\in T^{\ast }_{\boldsymbol X}(\mathcal{D}_{o}) \\
		\begin{array}{c}
			\displaystyle d_{L_1}\,\Big\downarrow \\
		\end{array}
		&  &
		\begin{array}{c}
			\displaystyle\Big\downarrow\, \displaystyle\frac{D}{Dt} \\
		\end{array}
		\\
		\begin{array}{c}
			\\
			\displaystyle\frac{D\boldsymbol C}{Dt}+\boldsymbol C\,\displaystyle\frac{\partial \boldsymbol{u}}{\partial
				\boldsymbol	x}\in T^{\ast }_{\boldsymbol x}(\mathcal{D}_{t})
		\end{array}%
		\ \ \ \ \ \  &
		\begin{array}{c}
			{\boldsymbol F}^{-1}
		 \\
			\longleftarrow \\
		\end{array}
		&
		\begin{array}{c}
			\ \ \ \ \ \ \ \  \\
			\ \ \ \displaystyle\frac{D\boldsymbol C}{Dt}\boldsymbol F+\boldsymbol C\, \displaystyle\frac{\partial \boldsymbol{u}}{%
				\partial \boldsymbol x} \boldsymbol F\in T^{\ast }_{\boldsymbol X}(\mathcal{D}_{o})%
		\end{array}%
	\end{array}
	\label{graph}
\end{equation*}
\bigskip

\noindent 
Consequently, the Lie derivative of $\boldsymbol{C}$ is 
\begin{equation*}
	d_{L_1} \boldsymbol C =  {\frac{{D \boldsymbol C}}{Dt}}+\boldsymbol C \,  \frac{\partial \boldsymbol u}{\partial \boldsymbol{x}}\equiv\frac{\partial \boldsymbol{C}}{\partial t}+\boldsymbol u^T\left(\frac{\partial \boldsymbol C^T}{\partial \boldsymbol{x}}\right)^T +\boldsymbol C \,  \frac{\partial \boldsymbol u}{\partial \boldsymbol{x}}.
\end{equation*}

A two-form field $\boldsymbol \omega \in T_{\boldsymbol x}^{2\star} (\mathcal D_t)$ has the inverse image $\boldsymbol \omega_0 \in T_{\boldsymbol X}^{2\star}  (\mathcal D_0)$.  There exists an isomorphism $\boldsymbol \omega(t, \boldsymbol{x})\in T_{\boldsymbol{x}}^{2\star}(\mathcal{D}_t) \longrightarrow \boldsymbol {w}(t, \boldsymbol{x})\in
T_{\boldsymbol{x}}(\mathcal{D}_t) $ defined as  
\begin{equation*}
	\forall\ \boldsymbol J_1\ \text{and}\  \boldsymbol J_2\in
	T_{\boldsymbol{x}}(\mathcal{D}_t),\quad \boldsymbol \omega\, (\boldsymbol J_1, \boldsymbol J_2)=\text{det}\,(\boldsymbol{w}, \boldsymbol J_1, \boldsymbol J_2).\label{2form}
\end{equation*}
We can define a vector field $\boldsymbol w_0$ such that:  
\begin{equation*}
	\text{det}\,\left(\boldsymbol{w}, \boldsymbol J_1, \boldsymbol J_2\right) = \text{det}\,(\boldsymbol{w_0}, \boldsymbol J_{10}, \boldsymbol J_{20}) 
\end{equation*}
where $\boldsymbol J_{10}$ and  $\boldsymbol J_{20}$ are the inverse image in $	T_{\boldsymbol{X}}(\mathcal{D}_0)$ of $\boldsymbol J_{1}$  and  $\boldsymbol J_{2}$ which verify \eqref{vectorfield}. Similarly  to the proof of lemma 1, we obtain 
\begin{equation*}
	\boldsymbol{w} \,(t, \boldsymbol{x}) =\frac{\boldsymbol{F}}{\text{det}\,\boldsymbol{F}} \,\boldsymbol w_0\,(t, \boldsymbol{X})\qquad {\rm or}\qquad \boldsymbol  w_0\,(t, \boldsymbol{X})=({\text{det}\,\boldsymbol{F}}) \,\boldsymbol{F}^{-1}	\boldsymbol{w} \,(t, \boldsymbol{x}).
\end{equation*}
The Lie derivative $d_{L_2}$ of  $\boldsymbol \omega (t,\boldsymbol x)$ is isomorphic  to  the image   of $\displaystyle\frac{\partial\boldsymbol w_0(t,\boldsymbol X)}{\partial t}$\,\ 
and we   deduce from \eqref{derF} 
\begin{equation*}
\text{det}\,\left(d_{L_2}{\boldsymbol{w},\boldsymbol{J}_1,\boldsymbol{J}_2}\right)= \text{det}\,\left(\frac{D\boldsymbol{ w}}{Dt}+\boldsymbol{w}\, \text{div}\,\boldsymbol u -\frac{\partial\boldsymbol u}{\partial\boldsymbol{x}}\, \boldsymbol{w},\boldsymbol{J}_1,\boldsymbol{J}_2\right), 
\end{equation*}
or
\begin{equation*}
d_{L_2}{\boldsymbol{w}}= \frac{D\boldsymbol{ w}}{Dt}+\boldsymbol{w}\, \text{div}\,\boldsymbol u -\frac{\partial\boldsymbol u}{\partial\boldsymbol{x}}\, \boldsymbol{w}.\label{Lie2}
\end{equation*}
Let us note that  $d_{L_1}\boldsymbol C$ and $d_{L_2}\boldsymbol w$ are zero if and only if $\boldsymbol C_0$ and $\boldsymbol w_0$ do not depend on time $t$ (they depend only on $\boldsymbol X$). A typical example of the two-form field is the vorticity vector. Another important example is the cross  product of two one-forms   $\ \displaystyle \frac{\partial \alpha}{\partial \boldsymbol x}\times\frac{\partial \beta}{\partial \boldsymbol x}$, where $ \alpha$ and $\beta$ are two scalar fields \cite{Gouin 2023}.

	\end{document}